\documentclass[11pt]{amsart} % use larger type; default would be 10pt

\usepackage{amssymb}

\usepackage[margin=1.25in]{geometry}

\usepackage{tikz-cd}
\numberwithin{equation}{section}

\usepackage{xcolor}
\definecolor{darkgreen}{rgb}{0,0.45,0}
\usepackage{ifpdf}
\ifpdf  \usepackage[pdftex,colorlinks,urlcolor=blue,citecolor=darkgreen,linkcolor=darkgreen,linktocpage]{hyperref}
\else   \usepackage[dvipdfmx,colorlinks,urlcolor=blue,citecolor=darkgreen,linkcolor=darkgreen,linktocpage]{hyperref}
\fi
\newcommand{\cat}[1]{\mathbf{#1}}
\newcommand{\twocat}[1]{\underline{\mathbf{#1}}}

\newcommand{\Hot}{\cat{Hot}}
\newcommand{\Set}{\cat{Set}}
\newcommand{\Gpd}{\cat{Gpd}}
\newcommand{\HoGpd}{\cat{HoGpd}}
\newcommand{\CC}{\cat{C}}
\newcommand{\K}{\cat{K}}

\newcommand{\mc}{\mathcal}
\newcommand{\GG}{{\mc G}}

\newcommand{\iso}{\cong}

\renewcommand{\subset}{\subseteq}

\newcommand{\id}{\mathrm{id}}
\newcommand{\op}{\mathrm{op}}

\newcommand{\gY}{Z_Y}
\newcommand{\gX}{Z_X}

\renewcommand{\a}{\alpha}
\renewcommand{\b}{\beta}

\renewcommand{\d}{\delta}
\newcommand{\ep}{\varepsilon}
\newcommand{\g}{\gamma}
\renewcommand{\k}{\kappa}
\renewcommand{\l}{\lambda}
\renewcommand{\th}{\theta}
\newcommand{\Z}{\mathbb{Z}}
\newcommand{\BZ}{\mathbf{B}\Z}

\newtheorem{thm}[equation]{Theorem}
\newtheorem{cor}[equation]{Corollary}
\newtheorem{lemma}[equation]{Lemma}
\newtheorem{prop}[equation]{Proposition}

\theoremstyle{definition}
\newtheorem{definition}[equation]{Definition}

\theoremstyle{remark}
\newtheorem{remark}[equation]{Remark}

\DeclareMathOperator*{\colim}{colim}

\newcommand{\sig}{\Sigma_\alpha^{\mathrm{c}}}

\newcommand{\define}[1]{\emph{#1}}

\title{Detecting isomorphisms in the homotopy category}

\author[Kevin Arlin]{Kevin Arlin\textsuperscript{\scriptsize 1}}
\thanks{\textsuperscript{\scriptsize 1}n\'e Carlson}
\address{Topos Institute, Berkeley, CA, USA}
\email{kevin@topos.institute}

\author{J. Daniel Christensen}
\address{University of Western Ontario, London, ON, Canada}
\email{jdc@uwo.ca}

\begin{document}


% This hack is no longer needed for me:
% To make 2020 show up:
%\renewcommand{\subjclassname}{%
%  \textup{2020} Mathematics Subject Classification}
% - These all exist in both 2010 and 2020:
% 55U35 Abstract and axiomatic homotopy theory in algebraic topology.
% 18A30 Limits and colimits (products, sums, directed limits, pushouts,
%       fiber products, equalizers, kernels, ends and coends, etc.)
% 55P65 Homotopy functors in algebraic topology.

\begin{abstract}
We show that no generalization of Whitehead's theorem holds for unpointed spaces.
More precisely,
we show that the homotopy category of unpointed spaces admits no set of objects
jointly reflecting isomorphisms. We give an explicit counterexample involving
infinite symmetric groups. In contrast, we prove that the spheres
do jointly reflect equivalences in the homotopy 2-category of spaces.
We also show
that homotopy colimits of transfinite sequential diagrams of spaces
are not generally weak colimits in the homotopy category, and furthermore
exhibit such a diagram with the property that none of its weak colimits 
is \emph{privileged}, which means, roughly, that it sees the spheres as compact objects.
The non-existence of a set jointly reflecting isomorphisms in the homotopy
category was originally claimed by Heller, but our results on weak colimits
show that his argument had an inescapable gap, leading to the need for
the new proof given here.
\end{abstract}

\date{July 13, 2023}

\maketitle

\section{Introduction}
\addtocounter{footnote}{1} % Because of the fake footnote on page 1.

Let $\Hot$ denote the homotopy category of spaces, and let $\Hot_{*,c}$ denote
the homotopy category of pointed, connected spaces.
Whitehead's theorem says that in $\Hot_{*,c}$, the set of spheres jointly
reflects isomorphisms.
One is naturally led to wonder whether there is a set of spaces in $\Hot$
which jointly reflects isomorphisms.

In~\cite{brown}, Brown proved that a functor $\Hot_{*,c}^\op \to \Set$ is
representable if and only if it is half-exact, in the sense that it sends
coproducts and weak pushouts in $\Hot_{*,c}$ to products and weak pullbacks in $\Set$.
In~\cite{heller}, Heller proved an abstract representability theorem:  if $\CC$ is a
category with coproducts and weak pushouts and $\CC$ contains a ``bounded" set $\GG$
of objects that jointly reflects isomorphisms (see Definition~\ref{def:introDefs} below), then a functor
$\CC^\op \to \Set$ is representable if and only if it is half-exact.
In the same paper, Heller gave an example of a half-exact functor $\Hot^{\op} \to \Set$
which is not representable.
He then claimed without proof~\cite[Prop.~1.2]{heller} that every set of spaces
in $\Hot$ is bounded,
and concluded~\cite[Cor.~2.3]{heller} that no set of spaces jointly reflects isomorphisms
in $\Hot$.

We show that it is not true that every set of spaces is bounded,
reopening the question of whether there is a set of spaces that jointly reflects
isomorphisms in $\Hot$.
We thus also give an independent proof that no set of spaces jointly reflects isomorphisms.
\medskip

We now give the definitions needed in order to precisely state our results.
\pagebreak

\begin{definition}\label{def:introDefs}
Let $\CC$ be any category and let $\GG \subset \CC$ be a set of objects.
\begin{enumerate}
 \item We say that $\GG$ \emph{jointly reflects isomorphisms} if a morphism
 $f: X \to Y$ in $\CC$ is an isomorphism whenever
 $\CC(S,f): \CC(S,X) \to \CC(S,Y)$ is a bijection for every $S \in \GG$.%
\footnote{Other terminology is in use, such as ``$\GG$ is a set of (weak) generators''
or ``the functors $\CC(S,-)$ are jointly conservative.''
Heller says that ``$\GG$ is left adequate.''}

 \item A \emph{weak colimit} of a diagram $D : \cat{I} \to \CC$
 is a cocone through which every cocone factors, not necessarily uniquely.

 \item A cocone $W$ of $D:\cat{I} \to \CC$ is \emph{$\GG$-privileged}
 if the canonical map
 \[
   \colim_{\a \in \cat{I}} \CC(S, D(\a)) \to \CC(S, W)
 \]
 is a bijection for every $S \in \GG$.

 \item For an ordinal $\b$, we say that $\GG$ is \emph{$\b$-bounded}
 if every diagram $D: \b \to \CC$ has a $\GG$-privileged weak colimit.

 \item We say that $\GG$ is \emph{left cardinally bounded}, or just \emph{bounded},
 if it is $\b$-bounded for each sufficiently large regular cardinal $\b$.
\end{enumerate}
\end{definition}

We use the word ``set'' to mean what is sometimes called a ``small set,'' i.e.,
an object of the category $\Set$.
All of our ordinals and cardinals are ``small.''
We regard a cardinal as an ordinal which is least in its cardinality class.
The \emph{cofinality} of an ordinal $\a$
is the smallest ordinal that is the order type of a cofinal subset of $\a$.
A cardinal is \emph{regular} if it is equal to its cofinality.

As mentioned above, $\Hot$ denotes the homotopy category of spaces,
by which we mean
the localization of the category of spaces at the weak homotopy equivalences, or equivalently,
the category whose objects are CW-complexes and whose morphisms are homotopy classes of continuous maps.
%Equivalently, we may consider $\Hot$ as the category of Kan complexes and simplicial homotopy classes of simplicial maps.
It is well-known that every small diagram in $\Hot$ has a weak colimit,
and that weak colimits are not unique.
\medskip

We can now state our main results more precisely.
First we give the result that shows that~\cite[Prop.~1.2]{heller} is false.

\theoremstyle{plain}
\newtheorem*{thm:HotLacksColims}{Theorem \ref{thm:HotLacksColims}}
\begin{thm:HotLacksColims}
The set $\GG = \{ S^n \mid n \geq 0 \}$ of spheres in $\Hot$ is not $\k$-bounded
for any ordinal $\k$ of uncountable cofinality.
That is, for each such $\k$, there exists a diagram $D: \k \to \Hot$ that admits no
$\GG$-privileged weak colimit.
\end{thm:HotLacksColims}

Note that Theorem~\ref{thm:HotLacksColims} applies to all uncountable regular cardinals,
showing that the set of spheres is not left cardinally bounded.
By adding one more space to the set, we can remove the uncountability assumption:

\newtheorem*{cor:HotLacksColims}{Corollary \ref{cor:HotLacksColims}}
\begin{cor:HotLacksColims}
Let $T$ denote a countably infinite, discrete space.
Then the set $\{ S^n \mid\break n \geq 0 \} \cup \{ T \}$ is not $\k$-bounded in $\Hot$
for any limit ordinal $\k$.
\end{cor:HotLacksColims}

The proof of Theorem \ref{thm:HotLacksColims} 
is somewhat involved and forms the bulk of the paper.
We first show that it is sufficient to find a counterexample in the
homotopy category $\HoGpd$ of groupoids.
Then, given $\k$ as in the statement, we consider the diagram \mbox{$D : \k \to \HoGpd$}
sending $\alpha$ to the free group on $2+\alpha$ generators.
We make use of the theory of graphs of groups~\cite{serre} and the
associated fundamental group\-oid~\cite{higgins} in order to construct a
sufficiently pathological cocone $D \to Z$ which we use to show that
$D$ admits no $\GG'$-privileged weak colimit, where $\GG' = \{\BZ\}$.
This involves a detailed understanding of the morphisms in $Z$ and
how they are expressed as words in the given generators.
It follows that the diagram $\k \to \Hot$ sending $\alpha$ to the
wedge of $\alpha$ circles has no $\GG$-privileged weak colimit, where $\GG$ is as
in the statement of Theorem \ref{thm:HotLacksColims}.
% We start at 2, not 0, but I think this form follows.

Heller's argument for his claim~\cite[Prop.~1.2]{heller} that any set $\mathcal G$
of objects in $\Hot$ is bounded
was to take the cocone $W$ to be the homotopy colimit, i.e., a generalized telescope.
Since such homotopy colimits are $\mathcal G$-privileged,
our result above implies that they are not, in general, even weak colimits
in $\Hot$. This is in contrast to the situation for
telescopes of sequences indexed by $\omega$, and for other homotopy colimits of
diagrams indexed by freely-generated categories.
In the introduction to~\cite{franke}, Franke suggests using a Bousfield-Kan
spectral sequence to show that Heller's claim is false,
by comparing weak colimits to homotopy colimits,
but we were unable to find an example in which we could prove that a
certain differential was non-zero.

In the homotopy category of pointed, connected spaces, the set of spheres 
jointly reflects isomorphisms---this is the classical form of Whitehead's theorem.
However, we conjecture that the set of spheres is not bounded in $\Hot_{*,c}$.
If this is true, it means that Heller's abstract representability theorem,
as stated, does not imply Brown's representability theorem.
That said, Heller's argument only requires a set of objects that jointly
reflects isomorphisms and is $\b$-bounded for \emph{some} regular cardinal $\b$.
% regular cardinal? does it matter?  Probably not, but there's no harm.
Thus, since the set of spheres is $\aleph_0$-bounded, the proof of
Heller's theorem goes through in $\Hot_{*,c}$.

\medskip

Next we state the result that shows that the statement of~\cite[Cor.~2.3]{heller} 
is nevertheless correct.

\newtheorem*{thm:noGenerator}{Theorem \ref{thm:noGenerator}}
\begin{thm:noGenerator}
The category $\Hot$ contains no set $\GG$ of spaces that jointly reflects isomorphisms.
That is, there exists no set $\GG$ of spaces such that, if $f: X \to Y$ is a map of spaces
and $f_*: \Hot(S,X) \to \Hot(S,Y)$ is a bijection for every $S \in \GG$,
then $f$ is an isomorphism in $\Hot$.
\end{thm:noGenerator}

This second result is easier to prove, and so we prove it first, in
Section~\ref{se:noGenerator}.
Our method is a generalization of~\cite[Proposition~4.1]{MMS}, which gives
a ``phantom homotopy equivalence,'' that is, a map in
$\Hot$ which while not an isomorphism is seen as one by all finite complexes.
Our proof also shows that there is no set of connected spaces that jointly
reflects isomorphisms in the homotopy category of connected spaces.
Moreover, Theorem~\ref{thm:noGenerator} implies similar
results in other settings.  For example, since $\Hot$ is a reflective subcategory
of the homotopy category of $(\infty,1)$-categories, it follows that there
is no set of $(\infty,1)$-categories that jointly reflects isomorphisms in that
category.

Since the $(\infty,1)$-category $\mc S$ of spaces certainly contains a set of 
objects jointly reflecting equivalences---namely the set whose only element is the one-point space---while 
its 1-categorical truncation $\cat{Hot}$ does not, one might ask which behavior
the $n$-categorical truncations of $\mc S$ exhibit for larger values of $n$. In fact, we show in
Theorem~\ref{thm:SpheresGenerateTwoHot} that in the
2-category $\twocat{Hot}$ of spaces, morphisms, and homotopy
classes of homotopies between them, the set of spheres \emph{does}
jointly reflect equivalences, which is the natural generalization of 
joint reflection of isomorphisms to 2-category theory. Intuitively, the 
reason for the divergent behavior of $\cat{Hot}$ and $\twocat{Hot}$ is
that the 2-morphisms of $\twocat{Hot}$ retain the information about 
based homotopies that is lost in $\cat{Hot}$. 

\medskip

\emph{Acknowledgments:} The first author would like to thank 
George Raptis for suggesting an argument that the spheres should generate 
$\twocat{Hot}$, simpler than that originally given for the tori. 
Both authors thank the referee for many valuable comments that helped to
improve the paper, including the citation to \cite{MMS} that now does the
bulk of the work in Section 4.

%%%%%%%%%%%%%%%%%%%%%%%%%%%%%%%%%%%%%%%%%%%%%%%%%%%%%%%%%%%%%%%%%%%%%%%%%%%%%
%%%%%%%%%%%%%%%%%%%%%%%%%%%%%%%%%%%%%%%%%%%%%%%%%%%%%%%%%%%%%%%%%%%%%%%%%%%%%
%section 2
%%%%%%%%%%%%%%%%%%%%%%%%%%%%%%%%%%%%%%%%%%%%%%%%%%%%%%%%%%%%%%%%%%%%%%%%%%%%%
%%%%%%%%%%%%%%%%%%%%%%%%%%%%%%%%%%%%%%%%%%%%%%%%%%%%%%%%%%%%%%%%%%%%%%%%%%%%%
\section{\texorpdfstring{$\Hot$}{Hot} admits no set that jointly reflects isomorphisms}
\label{se:noGenerator}

We make the following definitions.
For an ordinal $\alpha$, write $\Sigma_{\alpha}$ for the group of all
bijections of the set $\alpha$, ignoring order.
When $\beta < \alpha$, there is a natural inclusion
$\Sigma_{\beta} \hookrightarrow \Sigma_{\alpha}$,
and we define $\sig$ to be the union of the images of $\Sigma_{\beta}$
for all $\beta < \alpha$.
We typically consider $\sig$ when $\alpha$ is a cardinal,
considered as the smallest ordinal with that cardinality,
and we call the elements of $\sig$ \define{essentially constant} permutations.

\begin{thm}\label{thm:noGenerator}
The category $\Hot$ contains no set $\GG$ of spaces that jointly reflects isomorphisms.
(See Definition \ref{def:introDefs}.)
\end{thm}

\begin{proof}
Let $\GG$ be a set of spaces and let $\a$ be a regular cardinal
larger than the cardinality of $\pi_1(S, s_0)$ for each $S \in \GG$
and each $s_0 \in S$.
% If the CW complexes are finite and we don't require \a to be uncountable,
% we could take \a to be aleph_0.  But then it is not true that the
% fundamental groups must have cardinality < aleph_0!
% I'm pretty sure that when a CW complex has infinitely many cells,
% its fundamental group is no larger than the number of cells.
We must construct a map $f:X\to Y$ which is not a homotopy equivalence
but which induces bijections on homotopy classes of maps from spaces in 
$\GG$. 

Our example will be $Bs : B\sig \to B\sig$,
where $s:\sig\to\sig$ is the shift homomorphism given by 
\[
  (s\sigma)(\gamma) =
    \begin{cases}
      \sigma(\gamma')+1, & \gamma = \gamma'+1 \\
      \gamma,            & \gamma \text{ a limit ordinal,}
    \end{cases}
\]
for $\sigma \in \sig$.
(Here and in what follows, if $\gamma$ is a successor ordinal, we write $\gamma'$ for its predecessor.)
We must check that $s\sigma\in\sig$. First, it is essentially
constant: if $\b < \a$ and $\sigma$ fixes each $\gamma\geq\beta$, then 
for $\gamma > \b$ we have $(s\sigma)(\gamma)=\gamma$, if 
$\gamma$ is a limit ordinal, and $(s\sigma)(\gamma)=\sigma(\gamma')+1=\gamma'+1=\gamma$,
if $\gamma$ is a successor. 
Next, we see that $s$ is a homomorphism:
$s(\sigma\tau)$ and $(s\sigma) (s\tau)$ both fix all limit ordinals,
while for successors we have
\[(s\sigma)((s\tau) (\gamma))=\sigma([\tau(\gamma')+1]')+1=\sigma\tau(\gamma')+1=s(\sigma\tau)(\gamma) ,\]
as desired. Note that setting $\tau=\sigma^{-1}$, respectively
$\sigma=\tau^{-1}$, we confirm that $s\sigma$ is indeed a 
bijection.

Let $H$ be a group with classifying space $BH$ and let $X$ be a connected space. If $\mathbf{Gp}$ denotes the category of groups, recall that $\Hot(X,BH)$ is isomorphic to $\mathbf{Gp}(\pi_1(X),H)$ modulo conjugation by elements of $H.$
(See, for example, \cite[Corollary~V.4.4]{whitehead}.)
In particular, we have a natural isomorphism $\Hot(X, BH) \cong \Hot(B\pi_1(X), BH)$.
It also follows that for groups $G$ and $H$,
$\Hot(BG, BH)$ is isomorphic to $\mathbf{Gp}(G, H)$ modulo conjugation by elements of $H$,
and that an element of $\Hot(BG, BH)$ is a homotopy equivalence
if and only if it is represented by an isomorphism.

Note that $s$ is not surjective, since $s\sigma$ always preserves limit ordinals.
% It is injective.
Therefore, $Bs:B\sig\to B\sig$ is not a homotopy equivalence.
However, we will show that it induces an isomorphism on $\GG$.
First observe that it suffices to prove this for connected components of spaces in $\GG$.
It follows that it is enough
to prove this for spaces of the form $BG$, where $G$ is a group of cardinality less than $\alpha$. 
%(This uses that $\a$ is uncountable.)

Any map $BG\to B\sig$ arises from a homomorphism $\varphi:G\to \sig$, well-defined up to conjugation.
Since $\a$ is regular, there is a limit ordinal
$\beta < \alpha$ so that $\varphi(g) \in \Sigma_{\beta}$ for every $g\in G$.  We claim that 
$s \circ \varphi$ is conjugate to $\varphi$ by an element $\tau\in\sig$ defined as follows:
\[
  \tau(\gamma) =
    \begin{cases}
      \gamma',      & \gamma < \beta \text{ a successor ordinal}\\
      \beta+\gamma, & \gamma < \beta \text{ a limit ordinal}\\
      \gamma+1,     & \beta \leq \gamma < \beta + \beta\\
      \gamma,       & \text{otherwise.}
    \end{cases}
\]
It is straightforward to check that $\tau$ is a permutation, 
and it clearly
fixes ordinals greater than or equal to $\beta + \beta$, which is less 
than $\alpha$.
For $g \in G$, let $\sigma = \varphi(g)$.  Then, noting that
$\tau^{-1}(\gamma)=\gamma+1$ for any $\gamma<\beta$, 
we have 
\[
  \begin{aligned}
     (\tau^{-1} \sigma \tau)(\gamma)
  &= \begin{cases}
       \tau^{-1} (\sigma (\gamma')),        & \gamma < \beta \text{ a successor ordinal}\\
       \tau^{-1} (\sigma (\beta + \gamma)), & \gamma < \beta \text{ a limit ordinal}\\
       \tau^{-1} (\sigma (\gamma + 1)),     & \beta \leq \gamma < \beta + \beta\\
       \tau^{-1} (\sigma (\gamma)),         & \text{otherwise}\\
     \end{cases} \\
  &= \begin{cases}
       \tau^{-1} (\sigma (\gamma')), & \gamma < \beta \text{ a successor ordinal}\\
       \tau^{-1} (\beta + \gamma),   & \gamma < \beta \text{ a limit ordinal}\\
       \tau^{-1} (\gamma + 1),       & \beta \leq \gamma < \beta + \beta\\
       \tau^{-1} (\gamma),           & \text{otherwise}\\
     \end{cases} \\
  &= \begin{cases}
       \sigma (\gamma') + 1, & \gamma < \beta \text{ a successor ordinal}\\
       \gamma,               & \gamma < \beta \text{ a limit ordinal}\\
       \gamma,               & \beta \leq \gamma < \beta + \beta\\
       \gamma,               & \text{otherwise}\\
     \end{cases}\\
  &= s(\sigma)(\gamma).
  \end{aligned}
\]
We have used that if $\gamma \geq \beta$, then $\sigma(\gamma) = \gamma$,
and the consequence that if $\gamma < \beta$, then $\sigma(\gamma) < \beta$. 

In summary, we have shown that $Bs$ induces the identity on $\Hot(S,B\sig)$ for every $S \in \GG$,
proving the claim.
\end{proof}

\begin{remark}
Since the map $Bs: B\sig \to B\sig$ used in the proof has connected domain and
codomain, it follows that there is no set of connected spaces that jointly
reflects isomorphisms in the homotopy category of connected spaces.
\end{remark}

We explain the origin of the maps $s$ and $\tau$.
Morally, $s$ is conjugation by the successor operation on ordinals, with limit
ordinals handled specially.
The map $\tau$ implements this by ``making room'' for the relevant limit ordinals in a range
outside of the support of a particular permutation $\sigma$.
In fact, if we denote the map $\tau$ above by $\tau_\b$, then $s$ itself is conjugation by
$\tau_\a$ in $\Sigma_\gamma^{\mathrm{c}}$ for a regular cardinal $\gamma > \a$.

\begin{remark}
The referee pointed out an alternate proof of Theorem~\ref{thm:noGenerator},
which makes use of the techniques employed in~\cite[Lemma 2.2]{heller},
namely the use of HNN-extensions.
It also involves a map between classifying spaces, but is less explicit.
In addition, the referee and N.~Kuhn pointed out that the case when
$\alpha = \omega$ was proved in~\cite[Proposition~4.1]{MMS}, using an
approach very similar to the approach given here.
\end{remark}

%%%%%%%%%%%%%%%%%%%%%%%%%%%%%%%%%%%%%%%%%%%%%%%%%%%%%%%%%%%%%%%%%%%%%%%%%%%%%
%%%%%%%%%%%%%%%%%%%%%%%%%%%%%%%%%%%%%%%%%%%%%%%%%%%%%%%%%%%%%%%%%%%%%%%%%%%%%
%section 3
%%%%%%%%%%%%%%%%%%%%%%%%%%%%%%%%%%%%%%%%%%%%%%%%%%%%%%%%%%%%%%%%%%%%%%%%%%%%%
%%%%%%%%%%%%%%%%%%%%%%%%%%%%%%%%%%%%%%%%%%%%%%%%%%%%%%%%%%%%%%%%%%%%%%%%%%%%%

\section{The lack of privileged weak colimits}

In this section, we give an example showing that Heller's privileged weak colimits
do not generally exist.

\begin{thm}\label{thm:HotLacksColims}
The set $\GG = \{ S^n \mid n \geq 0 \}$ of spheres in $\Hot$ is not $\k$-bounded
for any ordinal $\k$ of uncountable cofinality, e.g., for any uncountable regular cardinal.
That is, for each such $\k$, there exists a diagram $D: \k \to \Hot$ that admits no
$\GG$-privileged weak colimit.
\end{thm}

In particular, $D$ admits no $\GG$-privileged weak colimit for any set $\GG$ containing
the spheres.
Note that the set of spheres is $\aleph_0$-bounded, so we learn that boundedness
for one ordinal does not imply it for ordinals with larger cofinality.
% Should we use \omega here and in intro?

\begin{cor}\label{cor:HotLacksColims}
Let $T$ denote a countably infinite, discrete space.
Then the set $\{ S^n \mid n \geq 0 \} \cup \{ T \}$ is not $\k$-bounded in $\Hot$
for any limit ordinal $\k$.
\end{cor}

\begin{proof}
Since $\k$ is a limit ordinal, it has infinite cofinality.
If $\k$ has uncountable cofinality, then Theorem~\ref{thm:HotLacksColims} applies.
If $\k$ has countable cofinality, then $\{ T\}$ is not $\k$-bounded.
\end{proof}

In Section~\ref{sse:spaces-to-groupoids}, we reduce the problem to finding
a counterexample in the homotopy category of groupoids.
In Section~\ref{sse:graphs-of-groups}, we recall the theory of graphs of
groups, and prove some general results about the word problem in the
fundamental groupoid of a graph of groups.
Finally, in Section~\ref{sse:groupoids}, we give a counterexample in the homotopy
category of groupoids and complete the proof of Theorem~\ref{thm:HotLacksColims}.

\subsection{Reducing from spaces to groupoids}\label{sse:spaces-to-groupoids}

To prove Theorem~\ref{thm:HotLacksColims} we will
work primarily in the homotopy category $\HoGpd$ of groupoids, that is,
the category of groupoids and isomorphism classes of functors. It is well known
that the geometric realization of groupoids induces a reflective embedding
$B:\HoGpd\to\Hot$ whose left adjoint is the fundamental groupoid functor
$\Pi_1$ and whose image consists of the $1$-types, i.e., the spaces $X$
with $\pi_n(X,x)=0$ for all $x\in X$ and $n>1.$
All this follows from the adjunction between $\pi_1$ and the classifying
space functor $B$ that was used in the proof of Theorem~\ref{thm:noGenerator}.

\begin{lemma}\label{ReductionToHog}
Suppose given a diagram $D:J\to \HoGpd$, a set $\GG'$ of groupoids,
and a set $\GG$ of spaces containing $B \GG'$ as well as $S^n$ for all $n$.
If $D$ admits no $\GG'$-privileged weak colimit in $\HoGpd$, then 
$B\circ D:J\to \Hot$ admits no $\GG$-privileged weak colimit in 
$\Hot$. 
\end{lemma}

\begin{proof}
We prove the contrapositive.
Let $\l:B\circ D\to X$ be a $\GG$-privileged weak colimit, with $X\in\Hot$.
Then, since left adjoints preserve weak colimits, $\Pi_1(\l) : D \to \Pi_1 X$ is
a weak colimit.  We will show that it is $\GG'$-privileged.

First, since $\l$ is $\GG$-privileged, every map $a:S^n\to X$ factors through
a 1-type $B D(j)$ for some $j$. Thus, when $n > 1$, $a$ is freely homotopic to a constant,
which implies that $\pi_n(X,x)$ is trivial for all $x\in X$. We conclude
that $X$ is a $1$-type itself, so that $X\simeq B\left(\Pi_1 X\right)$. 

Since $B$ is fully faithful, we see that $\Pi_1(\l): D\to \Pi_1 X$ is 
$\GG'$-privileged. Indeed, if $G\in \GG'$, then 
\begin{align*}
\HoGpd(G, \Pi_1 X)&\iso\Hot(BG,B\left(\Pi_1X\right))\iso \Hot(BG,X)\\
&\iso \colim_j \Hot(BG,BD(j)) \iso \colim_j \HoGpd(G,D(j)).
\end{align*}
One can show that the composite isomorphism is induced by $\Pi_1(\l)$.
\end{proof}

Thus it suffices to exhibit appropriately pathological diagrams in 
$\HoGpd$, and then to upgrade them to $\Hot$. We aim to give a
diagram in $\HoGpd$ admitting no weak colimit privileged with
respect to the set $\GG'=\{\BZ\}$.
Here $\BZ$ denotes the groupoid freely generated by an automorphism,
i.e., the groupoid with one object $*$ whose endomorphism group is the integers.
Of course, $B(\BZ)$ is homotopy equivalent to $S^1$,
so $\GG$ in Lemma~\ref{ReductionToHog} can be taken to be the set of spheres.

\begin{remark}\label{rem:s1vsautos}
Note that, for any groupoid $G$, a functor $f:\BZ\to G$ corresponds to an object
$f(*)$ of $G$ and an automorphism $f_*:f(*)\to f(*)$. Furthermore, two such
functors $f,g:\BZ\to G$ are naturally isomorphic if and only if the automorphisms
$f_*$ and $g_*$ are conjugate in $G$. In particular, a functor
$f:\BZ\to G$ factors through $h:H\to G$ in $\HoGpd$ if and only if $f_*$ is conjugate to 
an automorphism in the image of $h$. 
\end{remark}

\subsection{Graphs of groups}\label{sse:graphs-of-groups}

To construct our example,
we recall the notion of a graph of groups, and prove
Corollaries~\ref{cor:monotoneReduction} and \ref{cor:Xesembed}, and Lemma~\ref{lem:onFreeGroups}
that will be used in the next section.

\begin{samepage}

\begin{definition}\label{def:graphOfGroups}
A \emph{graph of groups} $\Gamma$ is given by:
\begin{itemize} 
\item A graph, i.e., a set $X$ of vertices, a set $Y$ of oriented edges,
functions $s, t: Y \rightrightarrows X$, and
an involution $\overline{(-)}: Y \to Y$ interchanging $s$ and $t$.
\item Groups $G_x$ and $G_y$ for $x\in X$ and $y\in Y$ equipped with monomorphisms
$\mu_y:G_y\to G_{s(y)}$ such that $G_y=G_{\bar y}$.
\end{itemize}
For simplicity, we assume that the groups $G_x$ are disjoint.
For more on graphs of groups, see~\cite[Section~I.5]{serre}
and~\cite[Section~1.B]{hatcher}.
\end{definition}

\end{samepage}

Higgins \cite{higgins} defined the fundamental
groupoid $\Pi_1\Gamma$ of a graph of groups. The groupoid
$\Pi_1\Gamma$ is the groupoid on objects $X$ with generating morphisms the elements of the 
groups $G_x$, endowed with $x$ as domain and codomain, together with the elements of $Y$ viewed 
as morphisms $y:s(y)\to t(y)$. These 
generators are subject to the relations holding in the groups $G_x$, as well as new relations
\[\mu_{\bar y}(a)=y\mu_y(a)\bar y,\] 
for every $y$ and every $a\in G_y$. 
Note in particular that $\bar y=y^{-1}$, and we shall use both notations. It may aid
the intuition to consider $\Pi_1\Gamma$ as the fundamental groupoid of the space built
from $\coprod_X BG_x$ with cylinders $BG_y\times I$ glued in for each set $\{y,\bar y\}$
of elements of $Y$ related by the involution.

By definition, the groupoid $\Pi_1\Gamma$ is a quotient of the groupoid $\K$ with
object set $X$ and with morphisms freely generated by
$(\coprod G_x) \coprod Y$, subject to the relations holding in the groups $G_x$.
A morphism $x_0\to x_n$ in $\K$ is given by a word $(a_n,y_n,\ldots,y_1,a_0)$, with $y_i\in Y$, $s(y_1)=x_0$, $t(y_n)=x_n$, and 
$s(y_{i+1})=t(y_i) =: x_i$ for $1 \leq i < n$, while $a_i \in G_{x_i}$ for $0 \leq i \leq n$.

The natural realization functor $\K \to \Pi_1\Gamma$ will be denoted by
$|(a_n,y_n,\ldots,y_1,a_0)|=a_n\circ y_n\circ \cdots \circ y_1\circ a_0$. Higgins proves that
every morphism of $\Pi_1\Gamma$ is \emph{uniquely} the image under $|\cdot |$ of 
a so-called ``normal'' word. We will not recall this concept, as we need only Higgins' 
corollary regarding the less rigid \emph{irreducible} words. 

A morphism $(a_n,y_n,\ldots,y_1,a_0)$ in $\K$ is called
\emph{reducible} if $n > 1$ and for some $i$, $y_{i-1}=\bar y_i$ and $a_{i-1} \in \mu_{y_i}(G_{y_i})$.
Otherwise, the morphism is said to be \emph{irreducible}.
% By this definition, $(e)$ *is* irreducible.  But this definition matches
% what was used in most places below.
% If we actually specified a path in the graph, then we could also consider
% the zero length paths, which would cleanly handle this case...
Note that a reducible word can be shortened by the move 
\[
 (\ldots, a_i, y_i,\mu_{y_i}(\hat a_{i-1}),\bar y_i, a_{i-2}, \ldots) \mapsto
 (\ldots, a_i \mu_{\bar y_i}(\hat a_{i-1}) a_{i-2}, \ldots)
\]
to a word with the same realization.
Therefore, every element of $\Pi_1 \Gamma$ is the realization of an irreducible word.
We will use a key result of \cite{higgins}.

\begin{prop}[{\cite[Corollary 5]{higgins}}]\label{thm:higgins}
Let $w$ be an irreducible word in $\K$.
If $|w|$ is an identity morphism in $\Pi_1\Gamma$,
then $w = (e)$, where $e$ is an identity element of some $G_x$.
\end{prop}
% This is precisely what Theorem 11 of Serre says in the group case, so
% we can just cite that if we switch to groups.

Define the \emph{length} $\ell(w)$ of the word $w = (a_n,y_n,\ldots,y_1,a_0)$ to be $n$.
We deduce the following:

\begin{cor}\label{cor:monotoneReduction}
Let $\Gamma$ be a graph of groups and consider a word $w$
in the groupoid $\K$.
If $\ell(w) > 0$ and $|w|$ is equal to the realization of a zero-length word,
then $w$ is reducible.
\end{cor}

\begin{proof}
Suppose that $w = (a_n, y_n, \ldots, y_1, a_0)$ for $n > 0$ and that
$|w| = |(a)|$ for some $a$ in some $G_x$.
Let $w' = (a_n, y_n, \ldots, y_1, a_0 \, a^{-1})$.
Then $|w'|$ is an identity morphism in $\Pi_1 \Gamma$, so by
Proposition~\ref{thm:higgins}, $w'$ is reducible.
Since reduction occurs at interior points, $w$ must be reducible as well.
\end{proof}

\begin{cor}\label{cor:Xesembed}
Given a graph of groups $\Gamma$ and a vertex $x$, the vertex group
$G_x$ embeds in the automorphism group of $x$ in the fundamental groupoid 
$\Pi_1 \Gamma$. 	
\end{cor}

Because of this, we regard elements of the vertex groups as elements of
the fundamental groupoid without explicitly naming the inclusion map.

\begin{proof}
The map sends $a \in G_x$ to the realization of the word $(a)$.
Since the word $(a)$ is irreducible, if the realization is an identity in $\Pi_1\Gamma$,
Proposition~\ref{thm:higgins} tells us that $a$ is the identity element of $G_x$.
Therefore, this map is injective.
\end{proof}

We next record some facts about free groups.
%, which are the
%fundamental groupoids of graphs of groups with $X=\{x\}$ a singleton and
%$G_x$ trivial. 

\begin{samepage}

\begin{lemma}\label{lem:onFreeGroups}
Let $A\subset B$ be nonabelian free groups, with $A$ free on generators
$\{a_i\}$ and $B$ free on $\{a_i\}\cup \{b_j\}$.
\begin{enumerate}
\item If $b\in B$ and for all $a\in A$ we have $b a b^{-1}=a$, then $b$ is the identity.
\item If $b\in B$ satisfies $b a b^{-1}\in A$ for some $a\in A$, then either $a$ is the
identity or $b\in A$. 
\end{enumerate}
\end{lemma}

\end{samepage}

\begin{proof}
Fix $b \in B$. For part (1), if we take $a = a_i$ then the assumption that $b a_i b^{-1} = a_i$
shows that an irreducible word for $b$ must have last letter $a_i$ or $a_i^{-1}$ for every $i$,
which is absurd since there are at least two $i$'s.

For part (2), we assume $a$ is nontrivial and $b \notin A$. Factor $b$ as
$b' b''$, where $b'' \in A$ while $b'$ is represented by an irreducible word with rightmost letter some $b_j$.
Then $b a b^{-1} = b' a' b'^{-1}$, where $a' := b'' a \, b''^{-1}$ is a
non-trivial element of $A$.  The conclusion
now follows from the observation that no reductions are possible in the concatenation
of the irreducible words for $b'$, $a'$ and $b'^{-1}$, since concatenating those words
gives no letter adjacent to its inverse.
\end{proof}

\subsection{A counterexample in the homotopy category of groupoids}\label{sse:groupoids}

We now apply the generalities above to the problem of weak colimits in
$\HoGpd$.

We fix for the rest of the paper an ordinal $\k$ of uncountable cofinality,
and introduce the main characters in our counterexample. Note that Theorem~\ref{thm:HotLacksColims}
will follow if we replace $\k=[0,\k)$ by the interval $[2,\k)$, since the two categories
are isomorphic.
We use the latter because it allows us to use simple indexing while ensuring
that all of the vertex groups below are non-abelian.

\begin{definition}\label{DandZ}
Define a graph of groups $\Gamma$ with object set $[2,\k)$, vertex group
 $G_\a$ free on $\a$ generators, edge 
set \mbox{$\{y_\a^\b:\b\to \a \mid\break \a\neq \b\in [2, \k) \}$}, and involution $y_\a^\b\mapsto y_\b^\a$. The edge group
$G_{y^\b_\a}$ is just $G_{\min(\b,\a)}$. The edge morphism
$\mu_{y^\b_\a}:G_{\min(\b,\a)}\to G_\b$ is the
natural inclusion. Let $Z=\Pi_1\Gamma$.

Next, define a diagram $D:[2,\k)\to \HoGpd$ by
letting $D(\a)=G_\alpha$,
with action on morphisms the natural inclusions, denoted by 
$D^\b_\a : D(\b) \to D(\a)$.
We have a cocone $A:D\to Z$ with $A_\a:D(\a)\to Z$ the natural 
inclusion of the vertex group. To see that these maps do constitute a cocone, we note
that $y_\a^\b$ is the unique component of a natural isomorphism 
$A_\b\iso A_\a\circ D^\b_\a$.
\end{definition}

We do not need this fact, but it may provide motivation to the reader to
know that $Z$ is the ``standard'' weak colimit of the diagram $D$, defined
as the homotopy coequalizer of the natural diagram
\[
  \coprod_{\b < \a} D(\b) \rightrightarrows \coprod_{\b} D(\b) .
\]
Critically,
we do \emph{not} have the relations $y^\b_\a y^\g_\b=y^\g_\a$ in $Z$ which would
allow us to lift $A$ into a cocone in the 2-category of groupoids. We now intend to show
that $D$ admits no privileged weak colimit by, roughly, showing that this failure is 
unavoidable: no choice of isomorphisms $A_\b\iso A_\a\circ D^\b_\a$ can give $A$
such a lift. 

Write $\gY$ for the subgroupoid of $Z$ generated by the edges of the graph.
Any morphism of $\gY$ can be uniquely written as a reduced word in the generators $y^\b_\a$.
We say that such a morphism \emph{passes through} a vertex $\a$ if this unique word
involves a generator with source or target $\a$.
The identity $\id_{\a}$ is said to \emph{pass through} $\a$ and no other vertex.

\begin{lemma}\label{lem:tIsHorizontal}
Let $u: \b \to \a$ in $Z$ and let $2 \leq \g \leq \min(\a, \b)$.
Then $u$ is in $\gY$ and does not pass through any vertex less than $\g$
if and only if
$u$ is the unique component of a natural isomorphism $A_\b \circ D^\g_\b
\iso A_\a \circ D^\g_\a$ between functors $D(\gamma)\to Z$. 
Explicitly, for all $a \in D(\gamma)$, we must have
$D^\g_\a(a) = u D^\g_\b(a) u^{-1}$ in $Z$.
\end{lemma}

\begin{proof}
Suppose that $u$ is in $\gY$ and does not pass through any vertex less than $\g$.
It suffices to show that $y^\b_\a$ conjugates $D^\g_\b$ into $D^\g_\a$
when $\g \leq \b \leq \a$.
In this case, $\mu_{y^\b_\a}$ is an identity map, and so the claim
follows from the defining relations of $Z$:
\[
    y^\b_\a \, D^\g_\b(a)  \, \bar{y}^\b_\a
  = y^\b_\a \, \mu_{y^\b_\a}(D^\g_\b(a)) \, \bar{y}^\b_\a
  =            \mu_{\bar{y}^\b_\a}(D^\g_\b(a))
  =            D^\b_\a(D^\g_\b(a))
  =            D^\g_\a(a) .
\]

For the converse,
let $u$ be the realization of an irreducible word $w=(a_n,y_n,\ldots,y_1,a_0)$. We
proceed by induction on $n$. If $n=0$, then $\a=\b$ and $u=|(a_0)|\in G_\b$. The 
assumption
that $D_\b^\g(a)=u D^\g_\b(a)u^{-1}$ shows that $u$ centralizes a nonabelian
subgroup of a free group. By Lemma~\ref{lem:onFreeGroups}~(1), we see that $u$ is trivial as desired.
And clearly $u$ does not pass through a vertex less than $\g$; indeed, it passes through only
$\beta$, and $\beta \geq \gamma$.

For the inductive step, assume $n > 0$. Then $s(y_1)=\b$ and $t(y_n)=\a$.
Let $t(y_1)=\d$, and note that $\d \neq \b$.
In terms of $w$, the assumption on $u$ is that the word
\[w'=(a_n,y_n,\ldots,y_1,a_0D^\g_\b(a)a_0^{-1},y_1^{-1},a_1^{-1},\ldots,y_n^{-1},a_n^{-1})\] 
has realization $D^\g_\a(a)$ for every $a\in G_\g$.
Thus, by Corollary~\ref{cor:monotoneReduction}, $w'$ is reducible.
Since by assumption $w$ is irreducible,
any reduction must occur at the central entry.
So, letting $\ep := \min(\b,\d)$, we must have $a_0 D^\g_\b(a) a_0^{-1} \in \mu_{y_1}(G_\ep) = D^\ep_\b(G_\ep)$.
In particular, $a_0 D^\ep_\b(\hat a) a_0^{-1} \in D^\ep_\b(G_\ep)$ for some
non-identity element $\hat a$ in $G_{\min(\g,\ep)}$.
So by Lemma~\ref{lem:onFreeGroups}~(2), we see that
$a_0\in D^\ep_\b(G_\ep)\subset G_\b $, that is, $a_0=D^\ep_\b(\hat{a}_0)$ for some $\hat{a}_0 \in G_{\ep}$.
It then follows that $D^\g_\b(a)$ is in the image of $D^\ep_\b$
for every $a \in G_\g$, which means that $\g \leq \ep$,
since the inclusions of vertex groups are strict.

The reduction of $w$ at its central entry is 
\[
  (a_n, y_n, \ldots, y_2,
  a_1 \, D^\ep_\d(\hat{a}_0) \, D^\g_\d(a) \, D^\ep_\d(a_0)^{-1} \, a_1^{-1},
  y_2^{-1}, a_2^{-1}, \ldots, a_n^{-1}).
\] 
Thus, if we define $u': \d \to \a$ to be $|w''|$,
where $w'' = (a_n, y_n, \ldots, y_2, a_1 D^\ep_\d(\hat{a}_0))$,
then $\ell(w'') < n$ and $u'$ conjugates $D^\g_\d$ to $D^\g_\a$.
By induction, $u' \in \gY$. Since
\[
  u' y_1 = a_n y_n \cdots y_2 a_1 D^\ep_\d(\hat{a}_0) y_1
         = a_n y_n \cdots y_2 a_1 y_1 D^\ep_\b(\hat{a}_0)
         = u,
\]
$u$ is in $\gY$ as well.
Finally, recall that we observed that $\g \leq \ep = \min(\b,\d)$.
By induction, $u'$ does not pass through any vertex less than $\g$.
So the same is true of $u = u' y_1$.
\end{proof} 

Let $\gX$ denote the subgroupoid of $Z$ containing those morphisms
in the image of $G_x$ for some $x$. By Corollary \ref{cor:Xesembed}, 
$\gX$ is isomorphic to the disjoint union of the groups $G_x$.

\begin{lemma}\label{lem:NoysinX}
Consider a morphism $z:\a\to\a$ in $Z$. If there are morphisms 
$u:\a\to \b$ and $v:\a \to\g$ in $Z$ such that $u z u^{-1}$ is in $\gX$ and $vzv^{-1}$
is in $\gY$, then $z=\id_\a$. 
\end{lemma}

\begin{proof}
Let $y=vzv^{-1}$. Note that the inclusion $\gY \to Z$ has a retraction
$r : Z \to \gY$ defined by sending the generators of each vertex group to identity
elements.
Since $uv^{-1} y vu^{-1}$ is in $\gX$, we have that $r(uv^{-1} y vu^{-1}) = r(uv^{-1})\, y \, r(uv^{-1})^{-1}$ is an identity,
and so $y$ is an identity. Since $y=vzv^{-1}$ is an identity, we have that $z$ is 
an identity as well.
\end{proof}

The following is the key technical result.

\begin{lemma}\label{IncoherenceOfZ}
Suppose given a family $u^\b_\a:\b\to\a$ of morphisms of $\gY$ for all $\b < \a \in [2, \k)$
such that $u^\g_\a = u^\b_\a u^\g_\b$ for all triples $\g < \b < \a$.
Then there exists a pair $\b < \a$ such that 
$u^{\b}_{\a}$ passes through some $\g$ with $\g<\b$. 
\end{lemma}

\begin{proof}
  Assume that this is not the case.
 Let $\d_0=2$ and $\d_1=3$.
 Inductively, for each $n\in\omega$ let $\d_n$ be an ordinal exceeding every vertex that $u^{\d_{n-2}}_{\d_{n-1}}$ passes through.
 This is possible because $\k$ is a limit ordinal. 
 
 For each $n$, $u^{\d_{n-1}}_{\d_n}$ can be written uniquely as a reduced word in the free groupoid $\gY$.
 Let $y_n$ be a letter in this word which is of the form $y^\b_\a$ with
 $\b < \d_n \leq \a$.  Such a letter must exist since $u^{\d_{n-1}}_{\d_n}$
 starts at a vertex less than $\d_n$ and ends at $\d_n$.
 Note that $y_n$ cannot occur in the reduced form of
 any $u^{\d_{k-1}}_{\d_{k}}$ with $k \neq n$.
 For $k < n$, this holds by definition of $\d_n$,
 and for $k > n$, this holds by our assumption that each
 $u^{\b}_{\a}$ only passes through $\g$ with $\g \geq \b$.
 In particular, the $y_n$'s are distinct.

 Using that $\k$ has uncountable cofinality, choose $\d_{\omega} < \k$
 to be an ordinal exceeding every $\d_n$.
 Consider the decompositions
\[
  u^{\d_0}_{\d_\omega} = u^{\d_1}_{\d_\omega} u^{\d_0}_{\d_1}
  = u^{\d_2}_{\d_\omega} u^{\d_1}_{\d_2} u^{\d_0}_{\d_1}
  = u^{\d_3}_{\d_\omega} u^{\d_2}_{\d_3} u^{\d_1}_{\d_2} u^{\d_0}_{\d_1} = \cdots
\] 
 In the expression $u^{\d_1}_{\d_\omega} u^{\d_0}_{\d_1}$, a $y_1$ occurs in
 the reduced form of the right-hand factor, and does not occur in the
 left-hand factor, so the reduced form of $u^{\d_0}_{\d_\omega}$ must
 contain a $y_1$.
 Similarly, the second decomposition involves a $y_2$, which can't be
 cancelled from either side, so the reduced form of $u^{\d_0}_{\d_\omega}$ must
 contain a $y_2$.
 Continuing, we see that the reduced form of $u^{\d_0}_{\d_\omega}$ must
 contain countably many distinct letters, a contradiction.
\end{proof}

Recall that $\k$ is an arbitrary ordinal of uncountable cofinality. 

\begin{prop}\label{HogLacksColims}
There exists a diagram $C:[2,\k)\to \HoGpd$ valued in the homotopy category of groupoids 
such that for any weak colimit with cocone $F:C\to W$, there exists an automorphism in $W$
which is not conjugate to any morphism in the image of any leg \mbox{$F_\a: C(\a)\to W$} of $F$. 
\end{prop}

\begin{proof}
We claim that the diagram $D$ (see Definition \ref{DandZ}) is an example of such a $C$.

Towards a contradiction, suppose $F:D\to W$ is a weakly colimiting cocone such that every 
automorphism in $W$ is conjugate to one in the image of some component of $F$.
Write $F_\a$ for functors representing the maps $D(\a) \to W$.
Since $F$ is a cocone in $\HoGpd$, for each $\b < \a \in [2, \k)$ we may 
choose a natural isomorphism
\[ h^\b_\a: F_\b \iso F_\a\circ D^\b_\a \]
between functors $D(\b) \to W$ in $\Gpd$. Denote by $\hat{h}^\b_\a$ the unique
component of $h^\b_\a$. As usual we shall denote $(h^\b_a)^{-1}$ by $h^\a_\b$,
and similarly for $\hat{h}$, as well as $u$ below.

Recall the natural cocone $A:D\to Z$ from Definition \ref{DandZ} 
and suppose given a representative $f : W \to Z$ of a 
factorization of the cocone $A$ through $F$. For each $\a$, pick a natural isomorphism $k_\a: A_\a \iso f \circ F_\a$ with unique component $\hat{k}_\a$.
For $\b < \a$, let $u^\b_\a = \hat{k}_\a^{-1} f(\hat{h}^\b_\a) \hat{k}_\b$,
the unique component of the natural transformation $A_\b\to A_\a \circ D^\b_\a$ defined by $(k_\a^{-1} * D^\b_\a) \circ (f*h^\b_\a) \circ k_\b$,
where $*$ denotes whiskering.\footnote{For instance, $f*h_\a^\b:f\circ F_\b \iso f\circ F_\a\circ D_\a^\b$ has unique component $f(\hat h^\b_\a)$.} 
By Lemma \ref{lem:tIsHorizontal}, we see that each 
 $u^\b_\a\in \gY$, so the same holds for the morphism $u_{\a\b\g}:\g\to \g$ defined as 
 $u^\a_\g u^\b_\a u^\g_\b$ for $\g < \b < \a$. Furthermore, the same lemma guarantees that 
 no $u^\b_\a$ passes through a vertex less than $\min(\b,\a)$. % We allow \a < \b

 For each $\g<\b<\a$, denote by $w_{\a\b\g} \in W$ the unique component of the 
 composite natural transformation 
\[h^\a_\g \circ (h^\b_\a * D^\g_\b) \circ h^\g_\b:F_\g\to F_\g.\]
We have 
$w_{\a\b\g} = \hat{h}^\a_\g \hat{h}^\b_\a \hat{h}^\g_\b$, so 
\[ \hat{k}_\g^{-1} f(w_{\a\b\g}) \hat{k}_\g = \hat{k}_\g^{-1} f(\hat{h}^\a_\g) \hat{k}_\a \hat{k}_\a^{-1} f(\hat{h}^\b_\a) \hat{k}_{\b} \hat{k}_\b^{-1} f(\hat{h}^\g_\b) \hat{k}_\g = u_{\a\b\g} . \]
In particular, $u_{\a\b\g}$ is conjugate to $f(w_{\a\b\g})$. 

On the other hand, by assumption on $F$, $w_{\a\b\g}$ is conjugate to a morphism in 
the image of some $F_\th:D(\th)\to W$, say to $F_\th(w'_{\a\b\g})$.
Composing with $f$, we see that $u_{\a\b\g}$ is conjugate to $f(F_\th(w'_{\a\b\g}))$.
Finally, using $\hat k_\g$, we see $u_{\a\b\g}$ is conjugate to $A_\th(w'_{\a\b\g})$, 
in particular, to an element of $Z_X$. Since we saw above that $u_{\a\b\g}$ is in $Z_Y$,
Lemma \ref{lem:NoysinX} shows that $u_{\a\b\g}=\id_\g$.

Finally, Lemma~\ref{IncoherenceOfZ} implies that at least one $u^\b_\a$ passes
through a vertex less than $\b$, contradicting what we saw above.
\end{proof}

\begin{proof}[Proof of Theorem \ref{thm:HotLacksColims}]
By Proposition~\ref{HogLacksColims} and Remark~\ref{rem:s1vsautos}, the
diagram $D$ admits no weak colimit privileged with respect to the set $\GG'=\{\BZ\}$.
Thus by Lemma~\ref{ReductionToHog}, $B \circ D$ admits no weak colimit in $\Hot$
which is privileged with respect to the set of spheres.
\end{proof}

\section{The spheres reflect equivalences in the 2-category of spaces}

We saw in Theorem~\ref{thm:noGenerator} that in the homotopy category of spaces
there is no set of objects that jointly reflects isomorphisms.
In this section, we show that in the homotopy \emph{2-category} of spaces,
the spheres do jointly reflect equivalences.
We first define the terms we are using.

\begin{definition}\label{def:TwoHot}
By $\twocat{Hot}$, we mean the 2-category
whose objects are spaces of the homotopy type of a CW-complex
and whose hom-categories are the fundamental
groupoids of mapping spaces, that is $\twocat{Hot}(X,Y)=\Pi_1(Y^X)$. 
\end{definition}

%to formally verify that this constitutes a 2-category, 
%we should note that Pi_1 is strong monoidal, so we're changing 
%enrichment from the simplicial category of such spaces to a
%2-category. Might be easier to think Kan complexes, but less natural
%perhaps for the argument below?

\begin{definition}\label{2-strong generator}
A set $\mc G$ of objects in a 2-category $\mc K$ 	
\define{jointly reflects equivalences} if, whenever $f:X\to Y$ 
is a morphism in $\mc K$ such that, for every $S\in \mc G$, the induced
functor $\mc K(S,f):\mc K(S,X)\to \mc K(S,Y)$ is an equivalence of categories,
then $f$ itself must be an equivalence in $\mc K$. 
\end{definition}

We shall show in Theorem \ref{thm:SpheresGenerateTwoHot} that the 2-category
$\twocat{Hot}$ admits a set $\mc G$ of objects that jointly reflects
equivalences, namely $\mc G=\{S^n \mid n \geq 0 \}$. 
Note that a map $f$ in $\twocat{Hot}$ is an equivalence if and only if it
is a homotopy equivalence. This theorem is a corollary of Theorem 1 in 
\cite{MMS}, which shows that for a map $f:X\to Y$ of (arcwise connected) spaces
which is surjective on all fundamental groups, bijectivity of $f$ on 
higher homotopy groups is equivalent to that on free homotopy classes of
maps from spheres.

With this, we are prepared to show that the spheres satisfy the analogue
of Whitehead's theorem for $\twocat{Hot}$.

\begin{thm}\label{thm:SpheresGenerateTwoHot}
The set $\mc G=\{S^n\}$ of spheres jointly reflects equivalences in the
$2$-category $\twocat{Hot}$ of spaces.
\end{thm}

\begin{proof}
Let $f:X\to Y$ be such that $\twocat{Hot}(S^n,f):\twocat{Hot}(S^n,X)\to \twocat{Hot}(S^n,Y)$ is an equivalence of groupoids, for every $n$. 
Consider an inclusion of $*$ into $S^0\cong *\sqcup *$.
Since this has a retraction, the functor $\twocat{Hot}(*,X)\to \twocat{Hot}(*,Y)$
is a retract of the equivalence $\twocat{Hot}(S^0,X)\to \twocat{Hot}(S^0,Y)$
and is therefore also an equivalence.
That is, $f$ induces an equivalence $\Pi_1(X)\to \Pi_1(Y)$ of fundamental groupoids.
Thus $f$ induces an isomorphism on $\pi_0$ and on every $\pi_1$.  

Therefore, we can apply Theorem 1 of \cite{MMS}, so that $f$ will 
be a homotopy equivalence as soon as it induces a bijection on free
homotopy classes of maps from $S^n$. Now, the set of free homotopy classes
of maps $S^n\to X$ is simply the set of connected components in the groupoid
$\twocat{Hot}(S^n,X)$. Since $f$ induces an equivalence $\twocat{Hot}(S^n,X)\to\twocat{Hot}(S^n,Y)$, \emph{a fortiori} it induces an isomorphism on 
connected components, and the theorem is proven.
\end{proof}

\bibliography{hotHasNoGenerator}{}
\bibliographystyle{plain}

\end{document}